\documentclass[11pt]{amsart}

\usepackage{amsmath,amsfonts,amssymb,amsthm}
\usepackage{txfonts}
\usepackage{latexsym,bm,graphicx}
\usepackage{mathrsfs}
\usepackage{hyperref}
\usepackage{color}
\usepackage{dsfont}
\usepackage{graphicx}   
\usepackage{hyperref}
\usepackage{accents}
\usepackage{enumerate}
\usepackage{xcolor}

\usepackage{verbatim}

\title[Partial regularity of harmonic maps from Alexandrov spaces]{Partial regularity of harmonic maps from Alexandrov spaces}
\author{Huabin Ge}
\address{H. Ge, School of Mathematics, Renmin University of China, Beijing, 100872, P.R. China,E-mail address: hbge@ruc.edu.cn }
\author{Wenshuai Jiang}
\address{W. Jiang, School of Mathematical Sciences, Zhejiang University, Hangzhou, 310058, P.R. China, E-mail address: wsjiang@zju.edu.cn}
\address{School of Mathematics and Statistics, The University of Sydney, NSW, 2006, Australia.
}
\author{Hui-Chun Zhang}
\address{H. C. Zhang, Department of Mathematics,  Sun Yat-sen University, Guangzhou 510275, P.R. China  E-mail address: zhanghc3@mail.sysu.edu.cn}

\newtheorem{thm}{Theorem}[section]
\newtheorem{prop}[thm]{Proposition}
\newtheorem{lem}[thm]{Lemma}

\newtheorem{conj}[thm]{Conjecture}

\theoremstyle{definition}
\theoremstyle{remark}

\newtheorem{defn}[thm]{Definition}
\newtheorem{rem}[thm]{Remark}

\numberwithin{equation}{section}

\newcommand{\ls}{\leqslant}
\newcommand{\gs}{\geqslant}
\newcommand{\wa}{\widetilde\angle}

\newcommand{\ip}[2]{\langle{#1},{#2}\rangle}

\newcommand{\R}{\mathbb{R}}
\newcommand{\M}{\mathbb{M}}

\newcommand{\E}{{\mathscr E}}

\begin{document}

\begin{abstract}
In this paper, we prove the Lipschitz regularity of continuous harmonic maps from an finite dimensional  Alexandrov space to a compact smooth Riemannian manifold. This solves a conjecture of F. H. Lin in \cite{lin97}. The proof extends the argument of Huang-Wang  \cite {hua-w10}.
\end{abstract}

\maketitle

\section{Introduction}

Gromov-Schoen \cite{gro-s92} initiated to study harmonic maps into singular spaces by the calculus of variation. A general theory of (variational) harmonic maps between singular spaces was developed by Korevaar-Schoen \cite{kor-s93}, Jost \cite{jos94,jos97} and Lin \cite{lin97}, independently.
 The regularity problem is a classical problem in the theory of harmonic maps, which has attracted the attention of many researchers. For the harmonic maps between smooth Riemannian manifolds, many regularity and singularity results have been established (see, for example, \cite{mor48, hel91,bet93,eva91,gia-h82,hil-kw77,sch83,sch-u82, CHN15,CN13,NV17,Lin99} and the survey \cite{har97} and the book \cite{lin-w08}).
The regularity of  harmonic maps from singular spaces (or manifolds with singular metric)  has also been developed extensively, such as \cite{che95,jos97,shi96,lin97,iw08,li-w13,zz18,zzz19}).

An Alexandrov space (with curvature bounded from below) is a metric space such that Toponogov comparison theorem of triangles holds locally \cite{bur-gp92,bur-bi01}. Some topological singularity may occur in an Alexandrov space.
 In this paper,  we are interested in the regularity of harmonic maps from an Alexandrov space with curvature bounded from below to a smooth Riemannian manifold. F. H. Lin \cite{lin97} first established a partial H{\"o}lder continuity for energy minimizing maps as follows.

\begin{thm}[F. H. Lin \cite{lin97}] \label{thm1.1}
Let $\Omega$ be a bounded open domain in an $n$-dimensional Alexandrov space with curvature bounded from below by $k$, and let $N$ be a compact smooth Riemannian manifold. Suppose $u$ is an energy minimizing map, then $u$ is locally H{\"o}lder continuous in $\Omega$ away from a relatively closed subset of Hausdorff dimension $\ls n-3.$
\end{thm}

Recall that the theory of regularity for harmonic maps between smooth manifolds includes two main steps: (i) to establish a (partial) H{\"o}lder continuity, and (ii) to improve the H{\"o}lder continuity to $C^{1,\alpha}$-regularity for some $\alpha\in(0,1)$. Considering the regularity of harmonic maps from Alexandrov spaces, based on the partial H\"older regularity of Theorem \ref{thm1.1},  F. H. Lin posted the following conjecture.
\begin{conj}[F. H. Lin \cite{lin97}] \label{conj1.2}
 The H\"older continuity can be improved to the Lipschitz continuity in Theorem \ref{thm1.1}.
\end{conj}

 In this paper, we will prove the following Lipschitz regularity.
\begin{thm}\label{thm1.3}
Let $\Omega$ and $N$ be as in Theorem \ref{thm1.1}. Then any continuous harmonic map  (need not to be an energy minimizer) must be locally Lipschitz continuous in $\Omega$. Precisely, there exists a constant $\epsilon=\epsilon(n,k,\Omega,N,\sup_N|A|)>0$ such that the following holds: If $u:\Omega\to N$ is a harmonic map and  a ball $B_{r_0}(x_0)\subset \Omega$, $r_0\ls1,$ such that $u$ is continuous on $B_{r_0}(x_0)$ and
\begin{equation}
{\rm osc}_{B_{r_0}(x_0)}u:=\sup_{x,y\in B_{r_0}(x_0)}d_N\big(u(x),u(y)\big)<\epsilon,
\end{equation}
 then  $u$ is Lipschitz continuous on $B_{\frac{r_0}{2}}(x_0)$  (with a Lipschitz constant depending  on $n,k,\Omega$, $r_0,\mu\big(B_{r_0}(x_0)\big)$, $\int_{B_{r_0}(x_0)}|\nabla u|^2{\rm d}\mu$ and $N$, and $\sup_N|A|$), where $A$ is the second fundament form of the isometrically embedding of $N$ into $\mathbb R^\ell$.

  Here and in the sequel of this paper, $\sup_E v$ means always ${\rm esssup}_Ev$, the essential supremum.
\end{thm}

Comparing with the H\"older estimate in Lin \cite{lin97} and Shi \cite{shi96},  they only used the fact the metric of Alexandrov space is locally $L^\infty$ in the regular point(see Subsection \ref{ss_regularpoint}). However, one can construct example (see Shi \cite{shi96} or \cite{Gi83}) to show that H\"older estimate is optimal if the coefficient of an elliptic operator is only $L^\infty$. One cannot expect the Lipschitz estimate for such operator.  An example in Chen \cite{che95} showed also that the H\"older continuity is optimal if the domain space has no a lower bound of curvature.
   Therefore, in order to show Theorem \ref{thm1.3}, we have to use more information about Alexandrov space. One important estimate to our proof is the Lipschitz estimate for harmonic function which is a special harmonic map to $\mathbb{R}$. We will use the Lipschitz estimate (see Section \ref{s_gradient_estimate}, see also \cite{zz12} ) for harmonic function several times in our proof and will use the fact that the smooth target $N$ could be able to embed isometrically into Euclidean space.

As a direct consequence  of the combination of Theorem \ref{thm1.1} and Theorem \ref{thm1.3}, we  solve Conjecture \ref{conj1.2} completely.

 \begin{thm}\label{thm1.4}
Let $\Omega,N$ and $u$ be as in Theorem \ref{thm1.1}. Then $u$ is locally Lipschitz continuous in $\Omega$ away from a relatively closed subset of Hausdorff dimension $\ls n-3.$
\end{thm}

\begin{rem}
In \cite{lin97} Lin posted a Lipschitz regularity conjecture of harmonic map from an Alexandrov space to a nonpositive curvature metric space. Such conjecture was completely solved by H. C .Zhang-X. P.  Zhu \cite{zz18} by constructing a nonlinear version of Hamilton-Jacobi flow for harmonic map.  Our proof of Theorem \ref{thm1.3} is independent of their techniques and results.
\end{rem}


Recalling the case when the domain space of the harmonic maps is smooth, Theorem \ref{thm1.3} has been proved in \cite{gia-h82,sch83}.  See  \cite{cha-wy99} for an elementary proof in this case. Recently, Huang-Wang \cite{hua-w10} provided another new proof in this case, based on Riesz potential estimate and the Green function on $\mathbb R^n$. Our proof of Theorem \ref{thm1.3} is an extension of Huang-Wang's argument but there is a subtle point. Since it is not known how to get a suitable regularity of Green functions on  a general Alexandrov space, then, we will prove a gradient estimate for the Poisson equations with a $L^1$-data, via a estimate of heat kernels on Alexandrov spaces, which is given in Sect. 3 (see Proposition \ref{prop3.2}).

In Sect. 2, we will collect some basic concepts  and  informations of analysis on  Alexandrov spaces.  In Sect. 4, we will provide some basic facts on harmonic maps on Alexandrov spaces. In the last section, we will give the proof of Theorem \ref{thm1.3}. A key step is to established a Morrey-type decay estimate (see Proposition \ref{p:Morreybound}).  \\

\noindent\textbf{Acknowledgements.}
H. Ge is partially supported by NSFC 11871094.  W. Jiang is partially supported by NSFC 11701507 and the Fundamental Research Funds for the Central Universities and ARC DECRA.   H. C. Zhang is partially supported by NSFC 11521101 and 11571374.

\section{Preliminaries}

 \subsection{Alexandrov spaces with curvature bounded below}$\  $

Let $(M,|\cdot\cdot\ |)$ be a complete metric space. It is called a \emph{geodesic space}
  if, for every pair points $p,q\in M$, there exists a point $r\in M$ such that $|pr|=|qr|=|pq|/2$.  Fix any $k\in\R$. Given three points $p,q,r$ in a geodesic space $M$, we can take a
  triangle $\triangle \bar p\bar q\bar r$ in $\M^2_k$ such that $|\bar p\bar q|=|pq|$,
  $|\bar q\bar r|=|qr|$ and $|\bar r\bar p|=|rp|$, where  $\mathbb M^2_k$ the simply connected, $2$-dimensional space form of constant sectional curvature $k$.  If $k>0$, we add the assumption  $|pq|+|qr|+|rp|<2\pi/\sqrt{k}$.
  We let $\wa_k pqr$ denote the angle at the vertex $\bar q$ of the triangle $\triangle \bar p\bar q\bar r$, and we call it a \emph{$k$-comparison angle}.

\begin{defn}\label{def2.1}
 Let $k\in \mathbb R$. A geodesic space $M$ is called an \emph{Alexandrov space with curvature bounded below by } $k$, denoted by $curv\gs k$, if it satisfies the
 following properties:\\
\indent(i) it is locally compact;\\
\indent(ii)  for any point $x\in M$, there exists a neighborhood $U$ of $x$ such that the following condition is satisfied:
 for any two geodesics $\gamma(t)\subset U$ and $\sigma(s)\subset U$ with $\gamma(0)=\sigma(0):=p$, the $k$-comparison angles
$\wa_\kappa \gamma(t)p\sigma(s)$
is non-increasing with respect to each of the variables $t$ and $s$.
\end{defn}

Let $M$ be an Alexandrov space with $curv\gs k$ for some $k\in\mathbb R.$ It is well known that the Hausdorff dimension of $M$ is always an integer or $+\infty$ (see,
 for example, \cite{bur-bi01,bur-gp92,ale-kp19}).
In the following, the terminology of ``an ($n$-dimensional) Alexandrov space $M$" means that $M$
 is an Alexandrov space with $curv\gs k$ for some $k\in \mathbb R$ and that its Hausdorff dimension $\dim_{\mathscr H}=n$.
We denote by $\mu:=\mathscr H^n$ the $n$-dimensional Hausdorff measure on $M$.  It holds the corresponding Bishop-Gromov inequality. Moreover, it holds the following local Alfors' regularity: For any bounded domain $\Omega$ in an $n$-dimensional Alexandrov space with $curv\gs k$, there exist two positive constants $C_1,C_2$, (depending on the diameter of $\Omega$ and $\mu(\Omega)$, if $k>0$, we add to assume that ${\rm diam}(\Omega)\ls \pi/(2\sqrt k),)$ such that
\begin{equation}\label{equ2.1}
 C_1\ls \frac{\mu(B_r(x))}{r^n}\ls C_2,\quad \forall \ x\in\Omega,\ \ 0\ls r\ls {\rm diam}(\Omega).
\end{equation}
Indeed, by the Bishop inequality (see \cite{bur-gp92}), we obtain the upper bound
\begin{equation*}
\mu(B_{r}(x))\ls \mu(B_{r}\subset \mathbb H^n(k))\ls C_{n,k,{\rm diam}(\Omega)}\cdot r^{n}
\end{equation*}
and the Bishop-Gromov inequality (see \cite{bur-gp92}) implies the lower bound
\begin{equation*}
\frac{\mu(B_{r}(x))}{\mu(\Omega)}\gs  \frac{\mu(B_{r}(x))}{\mu(B_{2\cdot {\rm diam}(\Omega)}(x))}\gs
\frac{\mu(B_{r}\subset \mathbb H^n(k)) }{\mu(B_{2\cdot {\rm diam}(\Omega)}\subset \mathbb H^n(k))}\gs    C'_{n,k,{\rm diam}(\Omega)}\cdot r^n.
\end{equation*}

On an $n$-dimensional Alexandrov space $M$, the angle between any two geodesics $\gamma(t)$ and $\sigma(s)$
 with $\gamma(0)=\sigma(0)\!:=p$ is well defined, as the limit
 $$\angle \gamma'(0)\sigma'(0)\!:=\lim_{s,t\to0}\wa_\kappa \gamma(t)p\sigma(s).$$
 We denote by $\Sigma'_p$ the set of equivalence classes of geodesic $\gamma(t)$ with $\gamma(0)=p$,
  where $\gamma(t)$ is equivalent to $\sigma(s)$ if $\angle\gamma'(0)\sigma'(0)=0$.
  $(\Sigma_p',\angle)$ is a metric space, and its completion is called the \emph{space of directions at} $p$, denoted by $\Sigma_p$.
  It is  known (see,
 for example, \cite{bur-bi01} or \cite{bur-gp92}) that $(\Sigma_p,\angle)$ is an Alexandrov space with curvature $\gs1$ of dimension $n-1$.
The \emph{tangent cone at} $p$, $T_p$,  is the Euclidean cone over $\Sigma_p$.  The ``scalar product" is given on $T_p$ by
\begin{equation*}
\ip{u}{v}:=\frac 1 2\Big(|uo|^2+|vo|^2-|uv|^2\Big),\quad \forall\ u,v\in T_p,
\end{equation*}
where $o$ is the vertex of $T_p$.
The $exponential \ map$  $\exp_p: W_p\subset T_p\to M$ is defined in the standard way. Generally speaking, the domain $W_p$ may not contain any neighborhood of $o$. This is one of the technical difficulties in Alexandrov geometry.

\begin{defn}[Boundary, \cite{bur-gp92}]\label{def2.2}
 The boundary of an Alexandrov space $M$ is defined inductively with respect to dimension.
 If the dimension of $M$ is one, then $M$ is a complete
Riemannian manifold and the \emph{boundary }of $M$ is defined as usual. Suppose that the dimension of $M$ is $n\gs2$.
A point $p$ is a \emph{boundary point} of $M$ if $\Sigma_p$ has non-empty boundary.
\end{defn}
\noindent\emph{From now on, we always consider Alexandrov spaces without boundary.}
We refer to the seminar paper \cite{bur-gp92} or the books \cite{bur-bi01,ale-kp19} for the details.

\subsection{Singularity and (almost) Riemannian structure}\label{ss_regularpoint} $\ $

 Let $k\in\mathbb R$ and let $M$ be an $n$-dimensional Alexandrov space with   $curv\gs k$. For any $\delta>0$, we denote
 $$ S_\delta:=\big\{x\in M:\ {\rm vol}(\Sigma_x)\ls (1-\delta)\cdot \omega_{n-1})\big\},$$
 where $\omega_{n-1}$ is the Riemannian volume of the standard $(n-1)$-sphere. $\ S_\delta$ is close  (see \cite{bur-gp92}).
 Each point $p\in S_\delta$ is called a $\delta$-\emph{singular point}.
The set
$$S_M:=\cup_{\delta>0}S_\delta$$
is called \emph{singular set}. A point $p\in M$ is called a \emph{singular point} if $p\in S_M$. Otherwise it is called a \emph{regular point}. Equivalently,
 a point $p$ is regular if and only if $T_p$ is isometric to $\mathbb R^n$ (\cite{bur-gp92}).  Since we always assume that the boundary of $M$ is empty, it is proved in \cite{bur-gp92} that the Hausdorff dimension of $S_M$ is $\ls n-2.$ We remark that Ostu-Shioya in \cite{ots-s94} constructed an Alexandrov space with nonnegative curvature such that its singular set is dense.

Some basic structures of Alexandrov spaces have been known in the following.
\begin{prop} \label{prop2.3}
Let $k\in\mathbb R$ and let $M$ be an $n$-dimensional Alexandrov space with $curv\gs k$.\\
(1) \ There exists a constant $\delta_{n,k}>0$ depending only on the dimension $n$ and $k$ such that for each $\delta\!\in\!(0,\delta_{n,k})$, the set
$ M\backslash S_{\delta}$ forms a Lipschitz manifold (\cite{bur-gp92}) and has a $C^\infty$-differentiable structure (\cite{kuw-ms01}).\\
(2)\  There exists a $BV_{\rm loc}$-Riemannian metric $g$ on $ M\backslash S_\delta$ such that \\
\indent $\bullet$ the metric $g$ is continuous in $M\backslash S_M$ (\cite{ots-s94,per-dc});\\
\indent $\bullet$ the distance function on $M\backslash S_M$ induced from $g$ coincides with the original one of $M$ (\cite{ots-s94});\\
\indent $\bullet$ the Riemannian measure on $M\backslash S_M$ induced from $g$ coincides with the Hausdorff measure of $M$ (\cite{ots-s94}).
\end{prop}

\subsection{Sobolev spaces and Laplacian on Alexandrov spaces} $\ $

Several different notions of Sobolev spaces on metric  spaces have been established,
 see\cite{che99,amb-gs14,kuw-ms01,sha00,kor-s93,haj-k00}.  They coincide with each other in the setting of Alexandrov spaces.

Let $M$ be an $n$-dimensional Alexandrov space with   $curv\gs k$ for some $k\in \mathbb R$. Let $\Omega$ be an open domain in $M$.
We denote by $Lip_{\rm loc}(\Omega)$ the set of locally Lipschitz continuous functions on $\Omega$, and by $Lip_0(\Omega)$
 the set of Lipschitz continuous functions on $\Omega$ with compact support in $\Omega.$

  For any $1\ls p\ls +\infty$ and
   $f\in Lip_{\rm loc}(\Omega)$, its $W^{1,p}(\Omega)$-norm is defined by
$$\|f\|_{W^{1,p}(\Omega)}:=\|f\|_{L^{p}(\Omega)}+\|{\rm Lip}f\|_{L^{p}(\Omega)},$$
where ${\rm Lip}f(x)$  is the \emph{pointwise Lipschitz constant} (\cite{che99})  of $f$ at $x$:
$${\rm Lip}f(x):=\limsup_{y\to x}\frac{|f(x)-f(y)|}{|xy|}.$$
 Sobolev space  $W^{1,p}(\Omega)$ is defined by the closure of the set of locally Lipschitz functions $f$ with  $\|f\|_{W^{1,p}(\Omega)}<\infty$
under  $W^{1,p}(\Omega)$-norm.
The space $W_0^{1,p}(\Omega)$ is defined by the closure of $Lip_0(\Omega)$ under  $W^{1,p}(\Omega)$-norm. We say a function $f\in W^{1,p}_{\rm loc}(\Omega)$ if $f\in W^{1,p}(\Omega')$ for every open subset $\Omega'\subset\subset\Omega.$ Here and in the following,
 ``$\Omega'\subset\subset\Omega$"
means $\Omega'$ is compactly contained in $\Omega$.
  The space $W^{1,p}(\Omega)$ is reflexible for any $1<p<\infty$ (see, for example,  Theorem 4.48 of \cite{che99}). For each $f\in W^{1,p}(\Omega)$, there exists a function $|\nabla f|\in L^p(\Omega)$ such that
  $\|f\|_{W^{1,p}(\Omega)}=\|f\|_{L^p(\Omega)}+\||\nabla f|\|_{L^p(\Omega)}.$

Fix a number $\delta\in (0,\delta_{n,k})$ sufficiently small (see Proposition 2.3). Recall that $M^*:=M\setminus S_\delta$ is a $C^\infty$-manifold.  Let $\Omega$ be an open set and denote $\Omega^*:=\Omega\setminus S_\delta=\Omega\cap M^*$. An important fact  is the following denseness result given in \cite{kuw-ms01}.
\begin{lem}[Kuwae et al. \cite{kuw-ms01}] \label{lem2.4}
Let $\Omega$ be  bounded open. For each $u\in W^{1,2}_0(\Omega)\cap L^\infty(\Omega)$, there exists $u_j\in Lip_0(\Omega^*)$ such that $u_j\to u$ in $W^{1,2}(\Omega)$ and $u_j\overset{*-weak}{\to}u$   in $L^\infty(\Omega)$, as $j\to\infty$.
\end{lem}
 \begin{proof}
 The convergence $u_j\overset{W^{1,2}(\Omega)}{\to} u$ as $j\to \infty$ is the assertion of Theorem 1.1 in \cite{kuw-ms01}.

  If $u\in L^\infty(\Omega)$, from the construction of $u_j$ in \cite{kuw-ms01}, we have $\|u_j\|_{L^\infty(\Omega)}\ls \|u\|_{L^\infty(\Omega)} $ for all $j\in\mathbb N$.
   (See Lemma 3.3 in \cite{kuw-ms01}, $u_j$ is taken by $\phi_j u$ for some suit cut-off functions). Thus, $\{u_j\}$ is $*$-weak compact in $L^\infty(\Omega)$. By combining with $u_j\overset{L^2}{\to}u$ as $j\to\infty$, this yields the second assertion.
  \end{proof}

 \begin{defn}[Distributional Laplacian]\label{def2.5}
The \emph{Laplacian} on $\Omega$ is an operator $\mathbf\Delta\ (=\mathbf\Delta_\Omega)$ on $W_{\rm loc}^{1,2}(\Omega)$ defined as the follows. For each function $f\in W_{\rm loc}^{1,2}(\Omega)$, its Lapacian $\mathbf\Delta f$ is a linear functional acting on  $Lip_0(\Omega)$ given by
\begin{equation}\label{equ2.2}
\mathbf\Delta f(\phi):=-\int_\Omega\ip{\nabla f}{\nabla \phi}{\rm d}\mu\qquad \forall\ \phi\in Lip_0(\Omega).\end{equation}
\end{defn}
\noindent This Laplacian (on $\Omega$) is linear and  satisfies the Chain rule and Leibniz rule (see \cite{ots-s94,kuw-ms01,gig15}).

 Fix any sufficiently small $\delta>0$. Thanks to Lemma 2.4,   it suffices  to take the test function $\phi\in Lip_0(\Omega\backslash S_\delta)$. If $f\in W^{1,2}(\Omega)$, then (\ref{equ2.2}) holds for all $\phi\in W^{1,2}_0(\Omega).$

 If, given  $f\in W_{\rm loc}^{1,2}(\Omega)$, there exists a function $u_f\in L^1_{\rm loc}(\Omega)$ such that
\begin{equation*}
\mathbf\Delta f(\phi)=\int_\Omega u_f\cdot\phi {\rm d}\mu\qquad \forall\ \phi\in Lip_0(\Omega),
\end{equation*}
then we   write as ``$\mathbf\Delta  f=u_f$ in the sense of distributions". It is similar for  $u_f\in L^p_{\rm loc}(\Omega)$ or $W^{1,p}_{\rm loc}(\Omega)$ for any $p\in[1,\infty]$.

A function $f\in W^{1,2}_{\rm loc}(\Omega)$ is call subharmonic if $\mathbf\Delta f\gs0$ in the sense of distributions, that is $\int_\Omega\ip{\nabla f}{\nabla \phi}{\rm d}\mu\ls 0$ for all $0\ls \phi\in Lip_0(\Omega).$ A basic fact is the maximum principle, which is well-known for experts (see, for example \cite[Theorem 7.17]{che99} or \cite{gri-h14}).  For the convenience of readers, we include a proof here.
\begin{lem}[Maximum principle] \label{lem2.6}
Let $f\in W^{1,2}(\Omega)$ be a subharmonic function such that $f-g\in W^{1,2}_0(\Omega)$ for some $g\in W^{1,2}(\Omega)\cap L^\infty(\Omega)$. Then
$$\sup_{\Omega}f\ls \sup_{\Omega}g.$$
 Here $\sup_\Omega$ means ${\rm esssup}_\Omega$.
\end{lem}
\begin{proof}
Let $f_H$ be the (unique) solution of Dirichlet problem $\Delta f_H=0$ such that $f_H-f\in W^{1,2}_0(\Omega).$ Thus, $f_H-g\in W^{1,2}_0(\Omega)$. By \cite[Theorem 7.8 and Theorem 7.17]{che99}, we get $f_H\in L^\infty(\Omega)$ and $\sup_\Omega f_H\ls \sup_\Omega g.$ Then, it suffices to show that $\sup_\Omega f\ls \sup_\Omega f_H$.

Notice that  $\mathbf\Delta (f-f_H)\gs0$ on $\Omega$ and $f-f_H\in W^{1,2}_0(\Omega)$. Then $(f-f_H)^+\in W^{1,2}_0(\Omega)$. Therefore, we get
 $$0\gs \int_{\Omega}\ip{\nabla (f-f_H)}{\nabla (f-f_H)^+}{\rm d}\mu=\int_\Omega|\nabla (f-f_H)^+|^2{\rm d}\mu.$$
Then $(f-f_H)^+=0$ almost everywhere, by Poincar{\'e} inequality. That is $f\ls f_H$ almost everywhere in $\Omega$. This yields $\sup_\Omega f\ls \sup_\Omega f_H$, and finishes the proof.
\end{proof}

\subsection{The heat flows on Alexandrov spaces} $\ $

Let $M$ be an $n$-dimensional Alexandrov space with $curv\gs k$ for some $k\in\mathbb R$. The Dirichlet energy $\mathscr E$ is defined by
\begin{equation*}
\mathscr E(f,g):=\int_M\ip{\nabla f}{\nabla g}{\rm d}\mu,\quad \forall\ f,g\in W^{1,2}(M).
\end{equation*}
This energy $\mathscr E$ gives a canonical Dirichlet form on $L^2(M)$ with the domain $D(\mathscr E)=W^{1,2}(M)$. It has
been shown \cite{kuw-ms01} that the canonical Dirichlet form $(\mathscr E, D(\mathscr E))$ is strongly local and that for  $f \in D(\mathscr E)$, the energy measure of $f$
is absolutely continuous w.r.t. $\mu$ with the density $|\nabla f |^2$. Moreover, the intrinsic distance $d_{\mathscr E}$ induced by $\mathscr E$ coincides with the original distance $d$.
48]).
Let $\Delta_{\mathscr E}$  and $\{P_tf\}_{t\gs 0}$ be the infinitesimal generator (with the domain $D(\Delta_{\mathscr E})$)
 and the heat flow induced from $(\mathscr E, D(\mathscr E))$.

It is proved in \cite{stu96,kuw-ms01} that there exists a locally H\"older continuous symmetric heat kernel $p_t(x,y)$ of $P_t$ such that
\begin{equation}\label{equ2.3}
P_tf(x)=\int_Mp_t(x,y)f(y){\rm d}\mu,\quad \forall\ f\in L^2(M).
\end{equation}
 Moreover, the following estimates for heat kernel have been proved in \cite{jia-lz16}.
\begin{lem} \label{lem2.7}
Let $M$ be an   $n$-dimensional Alexandrov space with $curv \gs k$ for some $k\in\mathbb R$.  There exist two constants $C_1,C_2>0$, depending only on $k$ and $n$, such that
\begin{equation}\label{equ2.4}
\begin{split}
\frac{1}{C_1\cdot\mu(B_{\sqrt t}(y))}\exp\Big(-&\frac{d^2(x,y)}{3t}-C_2\cdot t\Big)\ls p_t(x,y)\\
&\quad\ls\frac{C_1}{\mu(B_{\sqrt t}(y))}\exp\Big(-\frac{d^2(x,y)}{5t}+C_2\cdot t\Big)
\end{split}
\end{equation}
for all $t>0$ and all $x,y\in M$, and
\begin{equation}\label{equ2.5}
|\nabla  p_t(\cdot,y)|(x)\ls
\frac{C_1}{\sqrt t\cdot \mu(B_{\sqrt t}(x))}\exp\Big(-\frac{d^2(x,y)}{5t}+C_2\cdot t\Big)
\end{equation}
for all $t>0$ and $\mu$-almost all $x,y\in M$.
\end{lem}

We need the following result on cut-off functions, (which holds  on more general $RCD^*(K,N)$-spaces, see \cite[Lemma 3.1]{mon-n14} and  \cite{amb-ms15,gig-m15,hua-kx13}).
\begin{lem}\label{lem2.8}
Let $M$ be an   $n$-dimensional Alexandrov space with $curv \gs k$ for some $k\in\mathbb R$. Then for every $x_0\in M$ and $R>0$ there exists a Lipschitz cut-off function $\chi:  M\to [0,1]$ satisfying:\\
(i) $\ \chi=1$ on $B_{2R/3}(x_0)$ and ${\rm supp}(\chi)\subset B_{R}(x_0)$;\\
(ii) $\ \chi\in D(\Delta_{\E})$ and $\Delta_{\E}\chi\in W^{1,2}(M)\cap L^\infty(M)$, moreover $|\Delta_{\E}\chi|+ |\nabla \chi|\ls C(n,k,R)$.
\end{lem}

  It was shown \cite{gig15} that the above distributional Laplacian is compatible with the generator $\Delta_{\E}$ of the canonical Dirichlet form $(\mathscr E,D(\mathscr E))$ in the following sense:
\begin{equation}\label{equ2.6}
 f\in D(\Delta_{\E})\quad \Longleftrightarrow
\quad f\in W^{1,2}(M) \ {\rm and\ that }\ \mathbf\Delta f\  {\rm is\ a\ function\ in}\ L^2(M).
\end{equation}
Moreover, in this case, it holds $\mathbf\Delta f=\Delta_{\E}f$ in the sense of distributions.

\section{Gradient estimates to Poisson equations}\label{s_gradient_estimate}
We first give a gradient estimate of heat flows as follows.
\begin{lem} \label{lem3.1}
Let $n\gs 2$ and $k\in\mathbb R$, and let $M$ be an  $n$-dimensional Alexandrov space with $curv\gs k$. Let $F(t,x)\in L^2(M)$ for any $t\in[0,T]$. Suppose that there exist  $\bar F(x)\in L^2(M)$ such that
$$|F(t,x)|\ls \bar F(x)\ \ {\rm for\ all}\ \ t\in [0,T]\ {\rm \ and}\  \ \mu{\rm-a.e.}\  x\in M.$$
  Let $u(t,x)$ be a solution to the non-homogeneous heat equation
\begin{equation*}
\begin{cases}
\frac{\partial}{\partial t}u(t,x)=\Delta_{\mathscr E}u(t,x)+F(t,x)\\
u(0,x)=u_0(x)\in L^2(M).
\end{cases}
\end{equation*}
If both $u_0$ and $\bar F(x)$ are supported in $B_R(x_0)$.
 Then we have the gradient estimate of $u(t,x)$ as follows.
 \begin{equation}\label{equ3.1}
 |\nabla u(t,x)|\ls \frac{C_{n,k,R}}{ t^{\frac{n+1}{2}}}\cdot\fint_{B_R(x_0)}|u_0(y)|{\rm d}\mu(y)+C_{n,k,R}\fint_{B_R(x_0)} \frac{\bar F(y)}{d^{n-1}(x,y)}{\rm d}\mu(y)
 \end{equation}
 for almost all $x\in B_R(x_0)$ and all $t\in(0,\min\{T,R^2\})$, where $\fint_E f{\rm d}\mu:=\frac{1}{\mu(E)}\int_E f{\rm d}\mu$ for any measurable set $E$ and  $f\in L^1(E)$.
\end{lem}
  \begin{proof}
  From the Duhamel's principle (see Theorem 3.5 on Page 114 in \cite{paz83}), the solution $u(t,x)$ is unique and has a representation  as follows.
\begin{equation}\label{equ3.2}
\begin{split}
u(t,x)=\int_Mp_t(x,y)u_0(y){\rm d}\mu(y)+\int_0^t\int_Mp_{t-s}(x,y)F(s,y){\rm d}\mu(y){\rm d}s.
\end{split}
\end{equation}
 Since both $u_0$ and $\bar F$ are supported in $B_R(x_0)$, we have that for almost all $x\in B_R(x_0)$ and all $t\in(0,T]$
   \begin{equation}\label{equ3.3}
   \begin{split}
|\nabla u(t,x)|&\ls \int_M|\nabla p_t(x,y)||u_0(y)|{\rm d}\mu(y)+\int_0^t\int_X|\nabla p_{t-s}(x,y)||F(s,y)|{\rm d}\mu(y){\rm d}s\\
&\ls\int_{B_R(x_0)}|\nabla p_t(x,y)||u_0(y)|{\rm d}\mu(y)+\int_0^t\int_{B_R(x_0)}|\nabla p_{t-s}(x,y)|\bar F(y){\rm d}\mu(y){\rm d}s \\
&:=I_1+I_2.
\end{split}
\end{equation}
 From (\ref{equ2.5}) and the Bishop-Gromov inequality we get that for almost all $x\in B_R(x_0)$ and any $\sqrt t\ls R$,
 \begin{equation*}
 \begin{split}
 |\nabla p_t(\cdot, y)|(x)&\ls \frac{C_{1}\cdot e^{C_2R}}{\sqrt t\cdot \mu(B_{\sqrt t}(x))}\cdot \exp\Big(-\frac{d^2(x,y)}{5t}\Big)\\
 &\ls \frac{C_{3}}{\mu(B_R(x_0))}\cdot t^{-\frac{n+1}{2}}\cdot\exp\Big(-\frac{d^2(x,y)}{5t}\Big),
 \end{split}
 \end{equation*}
 where and in the sequel of this proof, all constants $C_1,C_2,\cdots$ depend only on $n,k$ and $R$.
 Hence, we obtain
 \begin{equation}\label{equ3.4}
 \begin{split}
 I_1&\ls \frac{C_{3}}{\mu(B_R(x_0))}\cdot\int_{B_R(x_0)} t^{-\frac{n+1}{2}}\cdot\exp\Big(-\frac{d^2(x,y)}{5t}\Big)|u_0(y)|{\rm d}\mu(y)\\
 &\ls \frac{C_{3}}{ t^{\frac{n+1}{2}}}\cdot\fint_{B_R(x_0)}|u_0(y)|{\rm d}\mu(y)\\
 \end{split}
  \end{equation}
and, by taking $\tau=\frac{d^2(x,y)}{t-s}$,
 \begin{equation}\label{equ3.5}
 \begin{split}
 I_2&\ls \frac{C_{3}}{\mu(B_R(x_0))}\cdot\int_{B_R(x_0)}{\rm d}\mu(y)\int^\infty_{d^2(x,y)/t} \Big(\frac{\tau}{d^2(x,y)}\Big)^{\frac{n+1}{2}}\cdot\exp\Big(-\frac{\tau}{5}\Big)\bar F(y)\frac{d^2(x,y){\rm d}\tau}{\tau^2}\\
 &\ls C_{3}\cdot\fint_{B_R(x_0)}\frac{\bar F(y){\rm d}\mu(y)}{d^{n-1}(x,y)}\int^\infty_{0} \tau^{\frac{n-3}{2}}\cdot\exp\Big(-\frac{\tau}{5}\Big){\rm d}\tau \\
  &\ls C_4\cdot\fint_{B_R(x_0)}\frac{\bar F(y){\rm d}\mu(y)}{d^{n-1}(x,y)},
 \end{split}
  \end{equation}
where we have used $n\gs 2$ and $\int_0^\infty \tau^{(n-3)/2}e^{-\tau/5}d\tau\ls C_n$.
  The desired estimate comes from the combination of (\ref{equ3.3}) and \eqref{equ3.4}, \eqref{equ3.5}.
   \end{proof}

 We need  the following local gradient estimate for Poisson equations.
\begin{prop} \label{prop3.2}
Let $n\gs 2$ and $k\in\mathbb R$, and let $M$ be an  $n$-dimensional Alexandrov space with $curv\gs k$.
  Let $u(x)\in W^{1,2}_{\rm loc}(B_R(x_0))$  solve  the Poisson equation
\begin{equation*}
\mathbf \Delta u=f  \in L_{\rm loc}^1(B_R(x_0))
\end{equation*}
in the sense of distributions.  If $u(x)\in  L_{\rm loc}^\infty(B_R(x_0))$  and if $f\gs0$,   then we have
 \begin{equation}\label{equ3.6}
 |\nabla u(x)|\ls   C_{n,k,R}\|u\|_{L^\infty(B_R(x_0))}+ C_{n,k,R} \fint_{B_R(x_0)} \Big(\frac{f}{d^{n-1}(x,y)}+|\nabla u|\Big){\rm d}\mu(y)
 \end{equation}
 for almost all $x\in B_{R/4}(x_0)$.
\end{prop}


\begin{proof}
(i) We first consider the case of $f\in L_{\rm loc}^2(B_R(x_0)).$

 Let $\phi: M\mapsto[0,1]$ be a cut-off function  such that $\phi=1$ on $B_{R/3}(x_0)$, $\phi=0$ out of $B_{2R/3}(x_0)$, and $R|\nabla \phi|+R^2|\Delta\phi|\ls C_{n,k,R}$ (see Lemma \ref{lem2.8}). Then it is clear  that
$u\phi\in W^{1,2}_0(B_R(x_0))\subset W^{1,2}(M)$ and
$$g:=f\phi+2\ip{\nabla u}{\nabla \phi}+u\cdot\Delta \phi\in L^2(M).$$
By (\ref{equ2.6}), we obtain that
$$u\phi\in D(\Delta_{\mathscr E})\quad {\rm and}\quad \Delta_{\mathscr E}(u\phi)=g.$$
By Lemma \ref{lem3.1}, we obtain, by taking $t=R^2$, for almost all $x\in B_{R/4}(x_0)$,
 \begin{equation}\label{equ3.7}
 \begin{split}
 |\nabla u|(x)= |\nabla (u\phi)|(x)\ls C_1  \fint_{B_R(x_0)}|u(y)|+ C_{1} \fint_{B_R(x_0)} \frac{|g(y)|}{d^{n-1}(x,y)}{\rm d}\mu(y),
\end{split}
 \end{equation}
where the constant $C_1$ depends only on $n,k$ and $R$. Noting that $|\nabla \phi|+|\Delta \phi|=0$ on $B_{R/3}(x_0)$, we have for  almost all $x\in B_{R/4}(x_0)$,
\begin{equation*}
\begin{split}
\int_{B_R(x_0)} \frac{|g(y)|}{d^{n-1}(x,y)}{\rm d}\mu(y)&\ls \int_{B_R(x_0)} \frac{|f|}{d^{n-1}(x,y)}{\rm d}\mu(y)\\
&\quad +\int_{B_R(x_0)\backslash B_{R/3}(x_0) } \frac{2|\nabla u|/R +|u|/R^2}{d^{n-1}(x,y)}{\rm d}\mu(y)\\
&\ls\int_{B_R(x_0)}\Big( \frac{|f|}{d^{n-1}(x,y)} +C_R(|\nabla u|+|u|)\Big){\rm d}\mu(y),
\end{split}
\end{equation*}
where we have used $d(x,y)\gs R/12$ provided $x\in B_{R/4}(x_0)$ and $y\in B_R(x_0)\backslash B_{R/3}(x_0)$. Hence, we get, for almost all $x\in B_{R/4}(x_0)$,
 \begin{equation}\label{equ3.8}
 \begin{split}
 |\nabla u|(x)\ls C_2\fint_{B_R(x_0)}|u|{\rm d}\mu + C_{2} \fint_{B_R(x_0)} \Big(\frac{|f|}{d^{n-1}(x,y)}+|\nabla u|\Big){\rm d}\mu(y)\\
  \ls C_2\ \|u\|_{L^\infty(B_R(x_0))} + C_{2} \fint_{B_R(x_0)} \Big(\frac{|f|}{d^{n-1}(x,y)}+|\nabla u|\Big){\rm d}\mu(y)\end{split}
 \end{equation}
where the constant $C_2$ depends only on $n,k$ and $R$.

 (ii) We consider the case of $0\ls f\in L^1_{\rm loc}(B_R(x_0)).$ Let $f_j(x)=\min\{f(x),j\}$. We have
 $0\ls f_j\in L^\infty\cap L^1(B_{3R/4}(x_0))\subset L^2(B_{3R/4}(x_0))$ and that $f_j\to f$ in $L^1(B_{3R/4}(x_0)),$ as $j\to\infty$.

 We solve the equation $\mathbf \Delta u_j=f_j$ on $B_{3R/4}(x_0)$ with $u_j-u\in W^{1,2}_{0}(B_{3R/4}(x_0))$ (in the sense of distributions). By $\mathbf\Delta(u-u_j)=f-f_j\gs 0$ in the sense
 of distributions. The maximum principle (Lemma \ref{lem2.6}) implies that, for almost all $x\in B_{3R/4}(x_0)$,
 \begin{equation}\label{equ3.9}
  {\rm esssup}_{B_{3R/4}(x_0)}(u-u_j)\ls0,\quad \forall\ j\in\mathbb N.
 \end{equation}
Similarly, the combination of the facts  $\mathbf\Delta u_j\gs0$, $u_j-u\in W^{1,2}_0(B_{3R/4}(x_0))$ and $u\in L^\infty(B_{3R/4}(x_0))$, by the maximum principle (Lemma \ref{lem2.6}) implies that
 $\sup_{B_{3R/4}(x_0)} u_j\ls {\sup}_{B_{3R/4}(x_0)}u,$ for each $j\in\mathbb N.$
 Then, we obtain, by combining with (\ref{equ3.9}),
  \begin{equation}\label{equ3.10}
 \|u_j\|_{L^\infty(B_{3R/4}(x_0))}\ls \|u\|_{L^\infty(B_{3R/4}(x_0))}\ \ \forall\ j\in\mathbb N.
 \end{equation}
  By  using $\mathbf\Delta(u_j-u)=f_j-f $ in the sense  of distributions, we have
\begin{equation}
\begin{split}
\int_{B_{3R/4}(x_0)}|\nabla (u_j-u)|^2{\rm d}\mu&\ls \int_{B_{3R/4}(x_0)}|(f_j-f)(u_j-u)|{\rm d}\mu\\
&\ls \|u_j-u\|_{L^\infty}\cdot\|f_j-f\|_{L^1}\to 0\ \ \ {\rm as}\ \ j\to\infty.
\end{split}
\end{equation}
 It follows  that $|\nabla u_j|\to |\nabla u|$ in $L^{2}(B_{3R/4}(x_0))$ as $j\to\infty.$   In particular, (up to a subsequence,)  we have $\lim_{j\to\infty}|\nabla u_j|(x)= |\nabla u|(x) $ at $\mu$-a.e. $x\in B_{3R/4}(x_0)$. By applying (\ref{equ3.8}) to $u_j$, and facts  (\ref{equ3.10}) and $0\ls f_j\ls f$, and then letting $j\to\infty$, we get
  \begin{equation*}
 |\nabla u(x)|\ls C_2\|u\|_{L^\infty(B_R(x_0))}+ C_{2} \fint_{B_R(x_0)} \Big(\frac{f}{d^{n-1}(x,y)}+|\nabla u|\Big){\rm d}\mu(y)
 \end{equation*}
 for almost all $x\in B_{R/4}(x_0)$. The proof is finished.
\end{proof}

\section{Harmonic maps from   Alexandrov spaces}

Let $(N,h)$ be a compact smooth Riemannian manifold. By Nash's imbedding theorem, we can assume that $N$
is isometrically embedded into an Euclidean space $R^\ell$.  Let $M$ be an $n$-dimensional Alexandrov space with $curv\gs k$ for some $k\in\mathbb R$. Fix any open domain $\Omega\subset M$,  the Sobolev space
$W^{1,2}(\Omega,N)$ is defined by
\begin{equation*}
W^{1,2}(\Omega,N):=\Big\{u\in W^{1,2}(\Omega, \mathbb R^\ell) \big| \
u(x)\in  N\ \ {\rm for}\ \mu{\rm-a.e.}\ x\in\Omega\Big\}
\end{equation*}
and the energy
\begin{equation}\label{equ4.1}
E(u):=\int_\Omega|\nabla u|^2(x){\rm d}\mu(x),\quad |\nabla u|^2=\sum_{j=1}^{\ell}|\nabla u^j|^2.
\end{equation}
\begin{defn}\label{def4.1}
A map $u \in W_{\rm loc}^{1,2}(\Omega,N)$ is a (weakly) harmonic map, if it is a critical point of $E(\cdot)$ on any subdomain $\Omega'\subset\subset \Omega.$
 In particular, any energy minimizing map is harmonic.
 \end{defn}

\begin{lem}\label{lem4.2}
Any harmonic map $u:\Omega\to N$ such that $u\in W^{1,2}(\Omega,N)$, then it satisfies the (weak) harmonic map system:
\begin{equation}\label{equ4.2}
\mathbf\Delta u=-A(\nabla u,\nabla u)
\end{equation}
in the sense of distributions, where $A(\cdot,\cdot)$ is the second fundament form of the embedding $N\subset \mathbb R^\ell.$ Namely,
\begin{equation*}
 \int_\Omega\ip{\nabla u}{\nabla \Phi} {\rm d}\mu=\int_\Omega A(\nabla u,\nabla u)\cdot\Phi {\rm d}\mu, \quad \forall\ \Phi\in W^{1,2}_0(\Omega,\mathbb R^\ell)\cap L^\infty(\Omega,\mathbb
 R^\ell).
 \end{equation*}
 Moreover, if  the image of $u$ is included in a geodesic ball $B_r(Q)$ for some point $Q\in N$ and $r<\frac{{\rm inj}(N)}{3},$ (${\rm inj}(N)$ is the injective radius of $N$,) then the function $u_Q(x):=d^2_N\big(Q,u(x)\big)$ satisfies
\begin{equation}\label{equ4.3}
0\ls \mathbf\Delta u_Q\ls  C_{0}\cdot |\nabla u|^2  \end{equation}
in the sense of distributions, for some constant $C_{0}>0$ depending on $N$, but independent of $Q$.
\end{lem}
\begin{proof}
(i) \ \ Let us first consider the case where $\Omega $ is a domain of a smooth manifold with an $L^\infty$-Riemannian metric. In this case   the assertion (\ref{equ4.2}) is well-known.  If the image $u(\Omega)\subset B_r(Q)$ with $r<{\rm inj}(N)/3$, then the function  $d_N^2(Q,\cdot)$ is smooth, and by the chain rule of harmonic maps (see \cite[Lemma 9.2.2]{jos17}), we have
\begin{equation}\label{equ4.4}
\mathbf \Delta u_Q={\rm Hess}_ud^2_N\big(Q,\cdot\big)(\nabla u,\nabla u)
\end{equation}
in the sense of distributions, where $ {\rm Hess}_ud^2_N(Q,\cdot)$ is the Hessian of $d^2_N(Q,\cdot)$ at $Q$ in $N$. The Hessian comparison theorem asserts
$$0\ls {\rm Hess}d^2_N\big(Q,\cdot)\ls C_{1}\cdot I$$
for some constant $C_1$ depending the bound of  $|sec_N|$ and ${\rm inj}(N)$, where the lower bound follows from the fact that $d_N^2(Q,\cdot)$ is convex. Hence, applying to  (\ref{equ4.4}), we obtain (\ref{equ4.3}) in this case.

\noindent (ii) \ \ Now we consider the general case where $\Omega$ is a domain of an Alexandrov space. From Proposition \ref{prop2.3}, we can fix a number $\delta\in (0,\delta_{n,k})$ sufficiently small such that  $M^*:=M\setminus S_\delta$ is a $C^\infty$-manifold.
Thus, the case (i) asserts that
\begin{equation}\label{equ4.5}
 \int_\Omega\ip{\nabla u}{\nabla \Phi} {\rm d}\mu=\int_\Omega A(\nabla u,\nabla u)\cdot\Phi {\rm d}\mu, \quad \forall\ \Phi\in Lip_0(\Omega\backslash S_\delta , \mathbb R^\ell).
\end{equation}
By Lemma \ref{lem2.4}, it is clear that the test (vector value) function $\Phi$ can be chosen in $W^{1,2}_0(\Omega,\mathbb R^\ell)\cap L^\infty(\Omega,\mathbb R^\ell)$. This is the assertion (\ref{equ4.2}).  The estimates (\ref{equ4.3}) can be obtained by a similar argument.    The  proof is finished.
\end{proof}

\section{From continuity to Lipschitz continuity}

In this section, we will prove Theorem \ref{thm1.3}. Throughout this section, we assume always that $\Omega$ is a bounded domain of an $n$-dimensional Alexandrov space with $curv\gs k$ for some $k\in\mathbb R$,  and that $(N,h)$  is compact smooth Riemannian manifold, isometrically embedded into $\mathbb R^\ell$, and let $u\in W^{1,2}(\Omega,N)$  be a harmonic map.

Remark that the  H\"older continuity of energy minimizing maps with small energy was established by Lin \cite{lin97}.

The following Morrey bound for $|\nabla u|$ is the key estimate (see also \cite{hua-w10} for a proof in Euclidean domain).
\begin{prop}\label{p:Morreybound}
Assume $\Omega, N$ as the above.  For any $0<\alpha<2$, there exists a constant $\epsilon>0$ (depending only on $\alpha, n,k,\Omega$ and the bound of the second fundamental form $\sup_N|A|$,)   such that the following holds: If $u:\Omega\to N$ is a harmonic map (need not to be an energy minimizer) and  if a ball $B_{r_0}(x_0)\subset \Omega$ such that
\begin{equation}
{\rm osc}_{B_{r_0}(x_0)}u:=\sup_{x,y\in B_{r_0}(x_0)}d_N\big(u(x),u(y)\big)<\epsilon,
\end{equation}
 then   for any $x\in B_{r_0/2}(x_0)$ and any $r\ls r_0/2$, we have
\begin{align}\label{e:Morreybound}
r^{2-n}\int_{B_r(x)}|\nabla u|^2(y){\rm d}\mu(y)\ls Cr^{\alpha},
\end{align}
where the constant $C$ depends on $\alpha,n,k,\Omega, r_0$ and $\int_{B_{r_0}(x_0)}|\nabla u|^2{\rm d}\mu$.
\end{prop}

\begin{proof}
  For any $\epsilon>0$ to be fixed later when fixing $\alpha$.  For each $r<r_0$, let us solve the  Dirichlet problem  $\mathbf\Delta v=0$ with $u-v\in W^{1,2}_0(B_r(x),\mathbb R^\ell)$  on $B_{r}(x)$. By maximal principle (see Lemma \ref{lem2.6}),  we have that
\begin{align}\label{e:oscillation1}
\sup_{y,z\in B_{r}(x)}|v(z)-v(y)|\ls \sup_{y,z\in B_r(x)}|u(z)-u(y)|\ls {\rm osc}_{B_{r_0}(x_0)}u\ls  \epsilon.
\end{align}
Thus by noting that $u-v\in W_0^{1,2}(B_r(x),\mathbb{R}^\ell)$, we have
\begin{align}\label{e:oscillation2}
\sup_{y\in B_{r}(x)}|(u-v)(y)|\ls  \sup_{y,z\in B_{r}(x)}|(u-v)(z)-(u-v)(y)|\ls  2\epsilon.
\end{align}
By using $(v-u)\in W^{1,2}_0\cap L^\infty$  and the equation $\mathbf\Delta (u-v)+A(u)(\nabla u,\nabla u)=0$ in the sense of distributions, we have
\begin{align*}
\int_{B_{r}(x)}\langle \nabla (u-v),\nabla (u-v)\rangle {\rm d}\mu +\int_{B_{r}(x)}A(u)(\nabla u,\nabla u)(v-u){\rm d}\mu=0.
\end{align*}
This implies, by (\ref{e:oscillation2}),
\begin{align}\label{equ5.5}
\int_{B_{r}(x)}|\nabla (u-v)|^2\ls  2\epsilon\cdot \sup_N|A|\cdot \int_{B_{r}(x)}|\nabla u|^2.
\end{align}
Recall that the Bochner inequality (see \cite{zz12}) implies  $|\nabla v|^2\in W^{1,2}_{\rm loc}(B_r(x))$ and that
$$\mathbf\Delta |\nabla v|^2\gs -2k|\nabla v|^2$$
 in the sense of distributions on $B_r(x)$. Thus, we have
\begin{equation}\label{equ5.6}
\sup_{B_{r/2}(x)}|\nabla v|^2\ls C_1\fint_{B_r(x)}|\nabla v|^2{\rm d}\mu(z),
\end{equation}
where  the constant $C_1$ depends only on $n,k$ and $r_0$.
For any $0<\theta<1/2$, by (\ref{equ5.5}), (\ref{equ5.6}) and that $v$ is harmonic on $B_r(x)$, we have
\begin{align}
\int_{B_{\theta r}(x)}|\nabla u|^2{\rm d}\mu&\ls 2\int_{B_{\theta r}(x)}|\nabla (u-v)|^2{\rm d}\mu+ 2\int_{B_{\theta r}(x)}|\nabla v|^2{\rm d}\mu\\
&\ls 2\int_{B_{r}(x)}|\nabla (u-v)|^2{\rm d}\mu+ 2\int_{B_{\theta r}(x)}|\nabla v|^2{\rm d}\mu\\
&\overset{(\ref{equ5.5})}{\ls} 4\epsilon\cdot \sup_N|A|\cdot \int_{B_{r}(x)}|\nabla u|^2+ 2\int_{B_{\theta r}(x)}|\nabla v|^2{\rm d}\mu\\
&\overset{(\ref{equ5.6})}{\ls} 4\epsilon\cdot \sup_N|A|\cdot \int_{B_{r}(x)}|\nabla u|^2+ 2\cdot C_1\mu(B_{\theta r}(x))\fint_{B_{ r}(x)}|\nabla v|^2{\rm d}\mu\\
\label{equ5.10}&\ls \left(4\epsilon \sup_N|A|+C_2\theta^n\right)  \int_{B_{r}(x)}|\nabla u|^2{\rm d}\mu,
\end{align}
where the constant $C_2$ depends on $n,k$ and $\Omega$, and we have used (\ref{equ2.1}) to conclude
$$\frac{\mu(B_{\theta r}(x))}{\mu(B_r(x))}\ls C_{n,k,\Omega}\cdot \theta^n.$$
Multiplying $(\theta r)^{2-n}$ to this inequality (\ref{equ5.10}), we get
\begin{align}\label{equ5.11}
(\theta r)^{2-n}\int_{B_{\theta r}(x)}|\nabla u|^2{\rm d}\mu\ls \left(4\epsilon \sup_N|A|\cdot \theta^{2-n}+C_2\theta^2\right) r^{2-n} \int_{B_{r}(x)}|\nabla u|^2{\rm d}\mu.
\end{align}
Given $0<\alpha<2$, let us choose $\theta_0=\theta(\alpha,C_2)<1/2$ such that $C_2\theta_0^2\ls 1/2 \theta_0^\alpha$. Then fix $\epsilon=\epsilon(n,\theta_0,\alpha,\sup_N|A|)$ such that $4\epsilon \sup_N|A| \theta_0^{2-n}\ls1/2 \theta_0^\alpha$. Therefore, we arrive at
\begin{align}\label{equ5.12}
(\theta_0 r)^{2-n}\int_{B_{\theta_0 r}(x)}|\nabla u|^2{\rm d}\mu\ls \theta_0^\alpha r^{2-n}\int_{B_r(x)}|\nabla u|^2{\rm d}\mu
\end{align}
for all $r\ls r_0$. By iterating this inequality finitely many times,  we have  for all $r\ls r_0$ that
\begin{align}\label{5.13}
r^{2-n}\int_{B_r(x)}|\nabla u|^2{\rm d}\mu\ls   r^\alpha\cdot C_{\alpha,\theta_0,r_0}\cdot\int_{B_{\theta_0r_0}(x)}|\nabla u|^2{\rm d}\mu.
\end{align}
This is enough to \eqref{e:Morreybound}, since that $\theta_0$ depends only on $\alpha,n,k$ and $\Omega$, and that $B_{\theta_0r_0}(x)\subset B_{r_0}(x_0)$.
\end{proof}

In order to dominate the  Riesz potentials, we have the following lemma.
\begin{lem}\label{lem5.2}
Let $n\gs 2$ and $k\in\mathbb R$, and let $M$ be an  $n$-dimensional Alexandrov space with $curv\gs k$.
   Given any ball $B_{r_0}(x_0)\subset M$ with $r_0\ls1$ and any $\beta\in(0,1)$,  there exists a constant $C(n,\beta)>0$  such that if $g\in L^1_{\rm loc}(M)$ then
 \begin{equation}\label{equ5.14}
\int_{B_{r_0/4}(x)}\frac{|g|(y)}{d^{n-1}(x,y)}{\rm d}\mu(y) \ls C_{n,\beta}\cdot  \sup_{r\ls r_0/2  }r^{\beta-n}\int_{B_r(x)}|g|(y){\rm d}\mu(y)
 \end{equation}
for $\mu$-a.e. $x\in B_{r_0/2}(x_0)$.
  \end{lem}
\begin{proof}
For any $0<s<r\ls r_0$, we denote the annulus $ A_{s,r}(x)=B_{r}(x)\backslash \bar B_{s}(x)$. For any $\varepsilon< r_0/4$, we have for all $x\in B_{r_0/2}(x_0)$ that
  \begin{equation}\label{equ5.15}
  \begin{split}
\int_{B_{r_0/4}(x)}\frac{|g|(y)}{d^{n-1}(x,y)}{\rm d}\mu(y) &= \int_{B_\varepsilon(x)}\frac{|g|(y)}{d^{n-1}(x,y)}{\rm d}\mu(y)+ \int_{A_{\varepsilon,r_0/4}(x)}\frac{|g|(y)}{d^{n-1}(x,y)}{\rm d}\mu(y)\\
& :=I_1+I_2.
\end{split}
 \end{equation}
 It is well-known  that
 \begin{equation}\label{equ5.16}
I_1 \ls C_{1} \cdot \varepsilon\cdot M_1g(x),
 \end{equation}
 where $M_1g(x):=\sup_{0< r\ls 1}\fint_{B_r(x)}|g|(y){\rm d}\mu(y)$ is the maximal function of $g$, and the constant $C_1$ depends on $n,k, B_{r_0}(x_0)$.  Indeed, we have
\begin{align*}
\int_{B_{\varepsilon}(x)}\frac{|g|(y)}{d(x,y)^{n-1}}{\rm d}\mu(y)&\ls \sum_{2^{-i}\ls \varepsilon} \int_{A_{2^{-i},2^{-(i-1)}}(x)}\frac{|g|(y)}{d(x,y)^{n-1}}{\rm d}\mu(y)\\
&\ls \sum_{2^{-i}\ls \varepsilon} 2^{i(n-1)}\int_{A_{2^{-i},2^{-(i-1)}}(x)}|g|(y){\rm d}\mu(y)\\
&\ls \sum_{2^{-i}\ls \varepsilon} 2^{i(n-1)}\mu\big(B_{2^{-(i-1)}}(x)\big)\fint_{B_{2^{-(i-1)}}(x)}|g|(y){\rm d}\mu(y)\\
&\ls C_2\sum_{2^{-i}\ls \varepsilon} 2^{n-i}\fint_{B_{2^{-(i-1)}}(x)}|g|{\rm d}\mu(y)\\
&\ls C_2\sum_{2^{-i}\ls \varepsilon} 2^{n-i}M_1g(x)\ls  C_22^n\cdot 2\varepsilon\cdot M_1g(x),
\end{align*}
where  we have used  (\ref{equ2.1}) and  $r_0\ls1$ to get
\begin{equation*}
\mu\big(B_{2^{-(i-1)}}(x)\big)\ls   C_2\cdot 2^{-(i-1)n},
\end{equation*}
for some constant $C_2$ depends on $n,k$ and $B_{r_0}(x_0).$ This yields (\ref{equ5.16}) by taking $C_1:=2^{n+1}C_2$.

Let us estimate $I_2$.
\begin{align*}
I_2&=\sum_{r_0/4\gs 2^{-i}\gs \varepsilon} \int_{A_{2^{-i},2^{-(i-1)}}(x)}\frac{|g |(y)}{d(x,y)^{n-1}}{\rm d}\mu(y)\\
&\ls \sum_{r_0/4\gs 2^{-i}\ge \varepsilon} 2^{i(n-1)}\int_{A_{2^{-i},2^{-(i-1)}}(x)}{|g|(y)}{\rm d}\mu(y)\\
&\ls \sum_{r_0/4\gs 2^{-i}\gs \varepsilon} 2^{i(n-1)}\int_{B_{2^{-(i-1)}}(x)}{|g|(y)}{\rm d}\mu(y)\\
\end{align*}
Hence
\begin{align*}
I_2&\ls \sum_{r_0/4\gs 2^{-i}\gs \varepsilon} 2^{i(n-1)}2^{(i-1)(\beta-n)}2^{-(i-1)(\beta-n)}\int_{B_{2^{-(i-1)}}(x)}{|g|(y)}{\rm d}\mu(y)\\
&\ls \sum_{r_0/4\gs 2^{-i}\ge \varepsilon}  2^{i(n-1)}2^{-(-i+1)(-n+\beta)}\cdot\sup_{r\ls r_0/2  }r^{\beta-n}\int_{B_r(x)}|g|(y){\rm d}\mu(y)\\
&=\sum_{r_0/4\gs 2^{-i}\ge \varepsilon}  2^{ n-\beta-(1-\beta)i}\cdot\sup_{r\ls r_0/2  }r^{\beta-n}\int_{B_r(x)}|g|(y){\rm d}\mu(y)\\
&\ls C_{n,\beta}\cdot \sup_{r\ls r_0/2  }r^{\beta-n}\int_{B_r(x)}|g|(y){\rm d}\mu(y)
\end{align*}
for any $\beta\in(0,1).$ By combining with (\ref{equ5.15}) and (\ref{equ5.16}), we obtain that   for all $\varepsilon<r_0/4$
 \begin{equation}\label{equ5.17}
\int_{B_{r_0/4}(x)}\frac{|g|(y)}{d^{n-1}(x,y)}{\rm d}\mu(y) \ls C_{1} \cdot \varepsilon\cdot M_1g(x)+C_{n,\beta}\cdot  \sup_{r\ls r_0/2  }r^{\beta-n}\int_{B_r(x)}|g|(y){\rm d}\mu(y).
 \end{equation}
 Since $M_1g(x)<\infty$ for  $\mu$-a.e. $x\in B_{r_0}(x_0)$ by the weakly $L^1$-boundedness of the maximal function,  we conclude by letting $\epsilon\to 0$ that
  \begin{equation}
\int_{B_{r_0/4}(x)}\frac{|g|(y)}{d^{n-1}(x,y)}{\rm d}\mu(y) \ls C_{n,\beta}\cdot  \sup_{r\ls1  }r^{\beta-n}\int_{B_r(x)}|g|(y){\rm d}\mu(y),
 \end{equation}
 for  $\mu$-a.e. $x\in B_{r_0/2}(x_0).$ The proof is finished.
\end{proof}

Now we are in the place to prove the main gradient estimate of $u$.

\begin{thm}\label{thm5.3}
Assume $\Omega, N$ as the above.  Then there exists a constant $\epsilon\in (0,\frac{{\rm inj}(N)}{10})$ (depending only on $n,k,\Omega$ and $\sup_N|A|$,)   such that the following holds: If $u:\Omega\to N$ is a harmonic map (need not to be an energy minimizer) and  if a ball $B_{r_0}(x_0)\subset \Omega$ with $r_0\ls1$ such that
$${\rm osc}_{B_{r_0}(x_0)}u:=\sup_{x,y\in B_{r_0}(x_0)}d_N\big(u(x),u(y)\big)<\epsilon,$$
 then for any point $Q\in N$ with $d_N\big(Q, u(B_{r_0}(x_0)\big)<\frac{{\rm inj}(N)}{5}$ and letting $u_Q(x)=d^2_N\big(Q,u(x)\big)$, we have
 \begin{equation}\label{equ5.19}
 {\rm esssup}_{B_{\frac{r_0}{32}}(x_0)} |\nabla u_Q|\ls C.
  \end{equation}
for some constant $C$ depending on  $n,k,\Omega$, $r_0,\mu\big(B_{r_0}(x_0)\big)$, $\int_{B_{r_0}(x_0)}|\nabla u|^2{\rm d}\mu$ and $N$, and $\sup_N|A|.$
 \end{thm}

\begin{proof}
Let us fix $\alpha=3/2$ in Proposition \ref{p:Morreybound}, we have
\begin{align}\label{equ5.20}
\sup_{r\ls r_0/2}r^{\frac{1}{2}-n}\int_{B_r(x)}|\nabla u|^2(y){\rm d}\mu(y)\ls C_1, \quad   \forall x\in B_{r_0/2}(x_0),
\end{align}
where the constant $C_1$ depends on $n,k,\Omega, r_0$ and $\int_{B_{r_0}(x_0)}|\nabla u|^2{\rm d}\mu$.

From Lemma \ref{lem5.2}, we have (by taking $\beta=1/2$ and $g=|\nabla u|^2$) that
\begin{equation}\label{equ5.21}
\int_{B_{r_0/4}(x)}\frac{|\nabla u|^2(y)}{d^{n-1}(x,y)}{\rm d}\mu(y) \ls C_{n,1/2}\cdot  \sup_{r\ls r_0/2  }r^{\frac 1 2-n}\int_{B_r(x)}|\nabla u|^2(y){\rm d}\mu(y)\overset{(\ref{equ5.20})}{\ls} C_2.
 \end{equation}
for $\mu$-a.e. $x\in B_{r_0/2}(x_0)$. In particular, there exists $x_1\in B_{\frac{r_0}{64}}(x_0)$ such that
$$\int_{B_{r_0/4}(x_1)}\frac{|\nabla u|^2(y)}{d^{n-1}(x_1,y)}{\rm d}\mu(y)\ls C_2.$$

Take any point $Q\in N$ such that
$$d_N\big(Q, u(B_{r_0}(x_0)\big)<\frac{{\rm inj}(N)}{5}.$$  We have $u(B_{r_0}(x_0)) \subset B_{\epsilon+ \frac{{\rm inj}(N)}{5}}(Q)\subset B_{\frac{{\rm inj}(N)}{4}}(Q)$.
By applying Proposition \ref{prop3.2} to the  estimates (\ref{equ4.3}) on $B_{r_0/4}(x_1)$, we obtain that
\begin{equation}\label{equ5.22}
\sup_{x\in B_{\frac{r_0}{16}}(x_1)}|\nabla u_Q|(x)\ls C_3\|u_Q\|_{L^\infty(B_{r_0/4}(x_1))}+\frac{C_4}{\mu\big(B_{r_0/4}(x_1)\big)}+C_{n,k,r_0}\fint_{B_{r_0/4}(x_1)}|\nabla u_Q|{\rm d}\mu
\end{equation}
for two positive constants $C_3$ and $C_4$, where $C_3$ depends on $n,k,r_0$, and $C_4$ depends on $C_2, \sup_{N}|A|,n,k,r_0$ and the constant $C_0$ in (\ref{equ4.3}).
The fact $u(B_{r_0}(x_0)) \subset B_{\frac{{\rm inj}(N)}{4}}(Q)$ implies
$$\|u_Q\|_{L^\infty(B_{r_0}(x_0))}\ls inj^2(N)/16:=C_5.$$
By noticing  that
$$|\nabla u_Q|=2 d_N\big(Q,u(\cdot)\big)|\nabla d_N\big(Q,u(\cdot)\big) |\ls 2 d_N\big(Q,u(\cdot)\big)|\nabla u |\ls 2\sqrt{C_5}|\nabla u|$$
 and by combining with    the fact $d(x_0,x_1)\ls \frac{r_0}{64}$, and the doubling property, we get from (\ref{equ5.22}) that
\begin{equation}\label{equ5.23}
\sup_{x\in B_{\frac{r_0}{32}}(x_0)}|\nabla u_Q|(x)\ls C_3\cdot C_5 +\frac{C_6}{\mu\big(B_{r_0}(x_0)\big)}+C_7\fint_{B_{r_0}(x_0)}|\nabla u|{\rm d}\mu\ls C_8,
\end{equation}
where the constant $C_8$ depends on $n,k,\Omega$, $r_0,\mu\big(B_{r_0}(x_0)\big)$, $\fint_{B_{r_0}(x_0)}|\nabla u|^2{\rm d}\mu$ and $N$, and $\sup_N|A|.$ The proof is finished.
 \end{proof}




Now we provide the proof of Theorem \ref{thm1.3} as follows.

\begin{proof}[Proof of Theorem \ref{thm1.3}] Let $\epsilon\in (0,\frac{{\rm inj}(N)}{10}) $ be given in Theorem \ref{thm5.3}.
 Fix any ball  $B_{r_0}(x_0)\subset \Omega$, $r_0\ls1$, such that ${\rm osc}_{B_{r_0}(x_0)}u<\epsilon$ and that $u$ is continuous on $B_{r_0}(x_0)$.
 From the gradient estimate of Theorem \ref{thm5.3}, we get that for any $Q\in N$ with
 $d_N\big(Q, u(B_{r_{0}}(x_0)\big)<\frac{{\rm inj}(N)}{5}$,
\begin{equation}\label{equ5.24}
 \frac{|u_Q(x)-u_Q(y)|}{|xy|}\ls C,\quad \mu{\rm -a.e.}\  x,y\in B_{\frac{r_{0}}{32}}(x_0),
 \end{equation}
where the  constant $C$ is independent of  $Q$  (but it may depend  on $n,k,\Omega,$ $r_0,\mu\big(B_{r_0}(x_0)\big)$, $\int_{B_{r_0}(x_0)}|\nabla u|^2{\rm d}\mu$ and $N$, and $\sup_N|A|$).
Since $u$ is continuous in $B_{r_0}(x_0)$, we have $u_Q$ is also continuous in  $B_{r_0}(x_0)$. Then, the inequality (\ref{equ5.24}) holds for all $x,y\in B_{\frac{r_0}{32}}(x_0)$. That is,
\begin{equation}\label{equ5.25}
|u_Q(x)-u_Q(y)|\ls C\cdot |xy|,\quad \forall\  x,y\in B_{\frac{r_{0}}{32}}(x_0).
\end{equation}

Take any $x,y\in B_{\frac{r_0}{32}}(x_0)$. Since $d_N\big(u(x),u(y)\big)<\epsilon<{\rm inj}(N)/10$, we can extend the geodesic $u(x)u(y)$ to  a point $Q$ such that $d_N\big(Q,u(x)\big)={\rm inj}(N)/6$ and
$$d_N\big(u(x),u(y)\big)=d_N\big(u(x),Q\big)-d_N\big(u(y),Q\big).$$
By applying (\ref{equ5.25}), we obtain
\begin{equation}
\begin{split}
d_N\big(u(x),u(y)\big)&=\frac{u_Q(x)-u_Q(y)}{d_N\big(u(x),Q\big)+d_N\big(u(y),Q\big)}\\
&\ls \frac{C\cdot |xy|}{d_N\big(u(x),Q\big)+d_N\big(u(y),Q\big)}\ls \frac{6C}{{\rm inj}(N)}\cdot |xy|,
\end{split}
\end{equation}
where we have used $d_N\big(u(x),Q\big)+d_N\big(u(y),Q\big)\gs {\rm inj}(N)/6$. It asserts that $u$ is Lipschitz continuous on $B_{\frac{r_0}{32}}(x_0)$ with Lipschitz constant $C_1:=6C/{\rm inj}(N)$.
The proof is finished.\end{proof}

\end{document}